\documentclass[letterpaper,11pt]{article}
\usepackage{amsmath,amsfonts,amssymb,amsthm,mathrsfs,relsize}
\newtheorem{thm}{Theorem}[section]
\newtheorem{lem}[thm]{Lemma}

\DeclareMathAlphabet{\mathbbmsl}{U}{bbm}{m}{sl}

\title{Star Saturation Number of Random Graphs\\[4mm]}

\author{
{\Large A. Mohammadian} \qquad   \qquad  {\Large B. Tayfeh-Rezaie}\\[2mm]
School of Mathematics,\\
Institute for Research in Fundamental Sciences\,(IPM),\\
P.O. Box 19395-5746, Tehran, Iran\\[2mm]
\textsf{ali\_m}@\textsf{ipm.ir}  \qquad  \qquad  \textsf{tayfeh-r}@\textsf{ipm.ir}\\[3mm]}

\date{}

\begin{document}

\maketitle

\begin{abstract}
For a  given graph   $F$, the $F$-saturation number of  a graph $G$ is the minimum  number of edges in  an  edge-maximal $F$-free subgraph of $G$.
Recently, the $F$-saturation number of  the Erd\H{o}s$\text{\bf--}$R\'enyi random graph $\mathbbmsl{G}(n, p)$ has been   determined asymptotically for any complete graph  $F$.
In this paper, we give an asymptotic formula for    the  $F$-saturation number of  $\mathbbmsl{G}(n, p)$  when   $F$ is   a  star graph.\\

\noindent {\bf AMS Mathematics Subject Classif{}ication\,(2010):}   05C35, 05C80.\\[2mm]
\end{abstract}

\section{Introduction}

All graphs in this paper  are assumed to be f{}inite,   undirected,  and without loops or multiple edges.
The vertex set and the edge set of a graph $G$ are  denoted by $V(G)$ and $E(G)$,  respectively. For any subset $S$ of $V(G)$, the induced subgraph of $G$ on $S$ is denoted by $G[S]$. For
an integer $n\geqslant1$ and a real number $p\in[0, 1]$, we denote by $\mathbbmsl{G}(n, p)$  the
probability space of all graphs on a f{}ixed vertex set of size $n$ where every two  distinct    vertices  are adjacent   independently with probability $p$.

In 1941, Tur\'an  posed one of the  foundational problems   in extremal graph theory \cite{tur}. His question was about the  maximum number
of edges in   a   graph on $n$ vertices  without  a copy of a given graph  $F$ as a subgraph, a  parameter  which is now denoted by $\text{\sl ex}(n, F)$.
A dual idea  called   `saturation number'  was introduced by Zykov \cite{zyk} and later independently by Erd\H{o}s, Hajnal,   and Moon \cite{erd}.
It  asks for  the minimum  number of edges in an edge-maximal $F$-free graph  on  $n$ vertices. We below present this  notion  in a more general form.

Fix a positive integer  $n$ and a graph $F$.
A graph  $G$ is   called {\it $F$-saturated} if  $G$ contains no subgraph   isomorphic to $F$    but
each graph obtained from $G$  by joining a pair of non-adjacent vertices
contains  at least one     copy of $F$   as a subgraph.  In other words,  $G$   is $F$-saturated if and only if it is an edge-maximal $F$-free graph. So,  $\text{\sl ex}(n, F)$ is equal to
the maximum number of edges  in  an $F$-saturated  graph  on   $n$ vertices.  The {\it saturation function} of $F$, denoted $\text{\sl sat}(n, F)$, is the minimum number of edges in  an $F$-saturated
graph on   $n$ vertices.
For instance,  it was  proved by Erd\H{o}s, Hajnal, and
Moon     \cite{erd}    that $$\text{\sl sat}(n,  K_r)=(r-2)n-\mathsmaller{{{r-1}\choose{2}}},$$ where $n\geqslant r\geqslant 2$ and $K_r$ is the complete graph on $r$ vertices.

For a given  graph $G$,  a spanning subgraph  $H$  of $G$  is  said to be an  {\it $F$-saturated subgraph of $G$} if
$H$ contains no subgraph   isomorphic to $F$
but each  graph obtained   by adding an  edge from  $E(G)\setminus E(H)$ to  $H$ has  at least one     copy of $F$   as a subgraph. The  minimum number of edges  in  an  $F$-saturated  subgraph of $G$ is  denoted by $\text{\sl sat}(G, F)$.  Thus, $\text{\sl sat}(n, F)$ is  by def{}inition equal to  $\text{\sl sat}(K_n, F)$. We refer the reader to   \cite{fau}  and the  references therein for a survey on graph saturation.

In recent years,   a new trend in   extremal graph theory has been developed to extend   the  classical results, such as Ramsey's  and Tur\'an's theorems, to random analogues. The study   reveals   the   behavior  of extremal parameters for a    typical graph. For instance,
Kor\'andi and   Sudakov   initiated  the study of graph saturation for random graphs    very recently  \cite{kor}.  They   proved  for every f{}ixed $p\in(0, 1)$ and f{}ixed  integer  $r\geqslant3$ that
$$\text{\sl sat}\big(\mathbbmsl{G}(n, p), K_r\big)=\big(1+o(1)\big)n\log_{\tfrac{1}{1-p}}n$$
with high probability. Let us recall    that, for a  sequence  $X_1, X_2, \ldots$ of random variables, we write `$X_n=o(1)$ with high probability' if   $$\mathlarger{\lim}_{n\rightarrow\infty}\mathbbmsl{P}\left(\left|X_n\right|\leqslant\epsilon\right)=1,$$
for any $\epsilon>0$.

Let $K_{1, r}$ be  the   star graph on $r+1$ vertices. In this paper, we investigate    the $K_{1, r}$-saturation number of  $\mathbbmsl{G}(n, p)$. The classical version   was resolved  by K\'aszonyi  and   Tuza    \cite{kas}, where they  proved  that
$$\text{\sl sat}(n, K_{1, r})=\left\{\begin{array}{ll}
{{r}\choose{2}}+{{n-r}\choose{2}}, &  \mbox{ if } r+1\leqslant n\leqslant\tfrac{3r}{2}\mbox{;}\\
\vspace{-1mm}\\
\left\lceil\tfrac{(r-1)n}{2}-\tfrac{r^2}{8}\right\rceil,  &   \mbox{ if } n\geqslant\tfrac{3r}{2}\mbox{.}\\
\end{array}\right.$$
The  f{}irst non-trivial case, namely   $r=2$,  is  especially  interesting  for the reason that        $\text{\sl sat}(G, K_{1, 2})$  is by def{}inition  equal to the minimum cardinality of a maximal matching in $G$.
It has been  proven by  Zito   \cite{zit} that
\begin{equation}\label{our}\mathlarger{\lim}_{n\rightarrow\infty}\mathbbmsl{P}\left(\tfrac{n}{2}-\log_{\tfrac{1}{1-p}}(np)<\text{\sl sat}\big(\mathbbmsl{G}(n, p), K_{1, 2}\big)<\tfrac{n}{2}-\log_{\tfrac{1}{1-p}}\sqrt{n}\right)=1.\end{equation}
Here we  show that with high probability
$$\text{\sl sat}\big(\mathbbmsl{G}(n, p), K_{1, r}\big)=\frac{(r-1)n}{2}-\big(1+o(1)\big)(r-1)\log_{\tfrac{1}{1-p}}n,$$
for every f{}ixed $p\in(0, 1)$ and  f{}ixed integer $r\geqslant2$. Note that,  for $r=2$, our result gives an upper bound  stronger than   \eqref{our} whereas our lower bound is weaker.
It is f{}inally  worth noting  that for    complete graphs    the saturation number of random graphs is much larger than the classical version  while the parameter   for star graphs   is  slightly   smaller  than its  classical value.

\section{Results}

Let $G$ be a graph and $k$ be a nonnegative   integer. A subset $S$ of $V(G)$  is called {\it $k$-independent}
if the maximum degree of $G[S]$ is at most  $k$. The  {\it $k$-independence number}  of $G$, denoted
by $\alpha_k(G)$, is def{}ined as the maximum cardinality of a $k$-independent set in  $G$. In particular, $\alpha_0(G)=\alpha(G)$ is the usual independence
number of $G$. The following  theorem is well known and  is   proved as    Theorem\,7.3 in    \cite{fri}.

\begin{thm}\label{matula}
{\rm (Matula \cite{mat})}
For any f{}ixed number $p\in(0, 1)$,
$$\alpha(\mathbbmsl{G}(n, p))=\big(2+o(1)\big)\log_{\tfrac{1}{1-p}}n$$
with high probability.
\end{thm}

The following easy observation can be proved using a straightforward union bound argument.
We apply  it to obtain    a generalized  version of Theorem \ref{matula}.

\begin{lem}\label{sudakov}
Let   $X$  be  a binomial random variable with parameters $n$ and $p\in(0, 1)$. Then $\mathbbmsl{P}(X\leqslant s)\leqslant{{n}\choose{s}}(1-p)^{n-s}$ for   any  $s\in\{0, 1, \ldots, n\}$.
\end{lem}

\begin{thm}\label{k}
For every f{}ixed number $p\in(0, 1)$ and f{}ixed integer $k\geqslant1$,
$$\alpha_k(\mathbbmsl{G}(n, p))=\big(2+o(1)\big)\log_{\tfrac{1}{1-p}}n$$
with high probability.
\end{thm}

\begin{proof}
Let $G\thicksim\mathbbmsl{G}(n, p)$, $q=1-p$, and $b=1/q$.
For any integer $s\geqslant1$, let $X_s$ be the number of induced subgraphs in $G$ on $s$ vertices with    at most $sk/2$ edges. Clearly,  $X_s=0$ implies  $\alpha_k(G)\leqslant s-1$.
For any  $S\subseteq V(G)$  with  $|S|=s$, let $Y_S$ count the number of edges  in $G[S]$. By Lemma \ref{sudakov},
\begin{align*}
\mathbbmsl{E}(X_s)&=\sum_{{S\subseteq V(G)}\atop{|S|=s}}\mathbbmsl{P}\left(Y_S\leqslant \tfrac{ks}{2}\right)\\
&\leqslant{{n}\choose{s}}{{{s}\choose{2}}\choose{\frac{ks}{2}}}q^{{{s}\choose{2}}-\frac{ks}{2}}\\
&\leqslant\left(\frac{n\text{\sl e}}{s}\right)^s\left(\frac{\text{\sl e}{{s}\choose{2}}}{\frac{ks}{2}}\right)^{\frac{ks}{2}}q^{{{s}\choose{2}}-\frac{ks}{2}}\\
&\leqslant\left(\left(\frac{n\text{\sl e}}{s}\right)^2\left(\frac{s\text{\sl e}}{k}\right)^kq^{s-k-1}\right)^{\frac{s}{2}}\\
&\leqslant\left(Cn^2s^kq^s\right)^{\frac{s}{2}},
\end{align*}
for some f{}ixed value  $C$.
Put $s=2\log_bn+2k\log_b\log_bn$. We have
$$\log_b\left(n^2s^kq^s\right)=2\log_bn+k\log_bs-s\longrightarrow-\infty$$
and so $n^2s^kq^s\rightarrow0$ as $n$ tends  to inf{}inity. Therefore,  $\mathbbmsl{E}(X_s)\rightarrow0$   and since  $\mathbbmsl{P}(X_s>0)\leqslant\mathbbmsl{E}(X_s)$ by the Markov    inequality, it follows that  $\mathbbmsl{P}(X_s>0)\rightarrow0$ as $n$ goes to inf{}inity. This proves that     $\alpha_k(G)\leqslant2\log_bn+2k\log_b\log_bn-1$ with high probability. Now, the assertion follows from the fact  $\alpha_k(G)\geqslant\alpha(G)$ and Theorem \ref{matula}.
\end{proof}

The following lemma is later used to prove the lower bound on $\text{\sl sat}(G, K_{1, r})$.

\begin{lem}\label{anyg}
For every graph $G$ on $n$ vertices   and    integer $r\geqslant2$,
$$\text{\sl sat}(G, K_{1, r})\geqslant\frac{(r-1)\big(n-\alpha_{r-2}(G)\big)}{2}.$$
\end{lem}

\begin{proof}
Let $H$ be a   $ K_{1, r}$-saturated   subgraph of   $G$. Let $A$ be the set of vertices of $H$  with   degree  at  most $r-2$ in $H$.  Since $H$ is  a  $ K_{1, r}$-saturated subgraph of   $G$, every vertex in  $\overline{A}=V(G)\setminus A$ is of degree $r-1$ in $H$ and   $G[A]=H[A]$. This implies that  $|A|\leqslant\alpha_{r-2}(G)$. We hence obtain that
\begin{equation*}
|E(H)|\geqslant\mathlarger{\mathlarger{\tfrac{1}{2}}}\sum_{v\in\overline{A}}\deg_H(v)\geqslant\frac{(r-1)\big(n-\alpha_{r-2}(G)\big)}{2}. \qedhere
\end{equation*}
\end{proof}

We will make use of the  next    theorem     in the  proof of our  main result.

\begin{thm}\label{alon}
{\rm (Alon$\text{\bf--}$F\"uredi \cite{alo})}
Let $G\thicksim\mathbbmsl{G}(n, p)$ be a random graph  and  $H$ be a f{}ixed   graph on $n$ vertices with maximum
degree $\mathnormal{\Delta}$, where $(\mathnormal{\Delta}^2+1)^2<n$. If
$$p^\mathnormal{\Delta}>\frac{10\log\left\lfloor\frac{n}{\mathnormal{\Delta}^2+1}\right\rfloor}{\left\lfloor\frac{n}{\mathnormal{\Delta}^2+1}\right\rfloor},$$
then the probability that $G$ does not contain a copy of $H$ is smaller than $1/n$.
\end{thm}

Now we  are in the position to prove our main result.

\begin{thm}
For every f{}ixed number $p\in(0, 1)$ and f{}ixed integer $r\geqslant2$,
$$\text{\sl sat}\big(\mathbbmsl{G}(n, p), K_{1, r}\big)=\frac{(r-1)n}{2}-\big(1+o(1)\big)(r-1)\log_{\tfrac{1}{1-p}}n$$
with high probability.
\end{thm}

\begin{proof}
Let $G\thicksim\mathbbmsl{G}(n, p)$, $q=1-p$, and $b=1/q$.
Using Theorem \ref{k} and Lemma \ref{anyg}, we f{}ind  that
$$\mathlarger{\lim}_{n\rightarrow\infty}\mathbbmsl{P}\left(\text{\sl sat}(G, K_{1, r})\geqslant\tfrac{(r-1)n}{2}-(1+\epsilon)(r-1)\log_bn\right)=1,$$  for any $\epsilon>0$.
So, it suf{}f{}ices to prove    that
\begin{equation}\label{1}\mathlarger{\lim}_{n\rightarrow\infty}\mathbbmsl{P}\left(\text{\sl sat}(G, K_{1, r})\leqslant\tfrac{(r-1)n}{2}-(1-\epsilon)(r-1)\log_bn\right)=1,\end{equation} for any $\epsilon>0$.
Fix $\epsilon$ and let  $\ell$ be  the least  integer such that   $\ell\geqslant(2-2\epsilon)\log_bn$ and  $(n-\ell)(r-1)$ is even. Also, f{}ix a regular graph $L$ on   $n-\ell$ vertices with   degree  $r-1$. For any   $S\subseteq V(G)$ with $|S|=\ell$, let
$$X_S=\left\{\begin{array}{ll}
1, &  \mbox{\small if  $S$ is an     independent set in $G$  and}\\
\vspace{-4.5mm}\\
&  \mbox{\small  $G[V(G)\setminus S]$ has a copy of $L$ as a subgraph;}\\
\vspace{-0.5mm}\\
0,  &   \mbox{\small otherwise.}\\
\end{array}\right.$$
We assume $n$ to be large enough whenever needed. It follows from  Theorem \ref{alon} that  $\mathbbmsl{E}[X_S]\geqslant q^{{\ell}\choose{2}}\left(1-\tfrac{1}{n-\ell}\right)$. Therefore, if we let $$X=\mathlarger{\mathlarger{\sum}}_{{S\subseteq V(G)}\atop{|S|=\ell}}X_S,$$ then  $\mathbbmsl{E}[X]\geqslant{{n}\choose{\ell}}q^{{\ell}\choose{2}}\left(1-\tfrac{1}{n-\ell}\right)$.
Moreover, for every subsets $S, T\subseteq V(G)$ of size $\ell$ with $|S\cap T|=i$, we easily see that   $\mathbbmsl{E}[X_SX_T]\leqslant q^{2{{\ell}\choose{2}}-{{i}\choose{2}}}$.
By the Chebyshev  inequality and noting that $n-\ell$ goes to inf{}inity, we have
\begin{align*}
\mathbbmsl{P}(X=0)&\leqslant\frac{\mathrm{var}(X)}{\mathbbmsl{E}[X]^2}\\
&=\mathlarger{\mathlarger{\sum}}^{}_{{S, T\subseteq V(G)}\atop{|S|=|T|=\ell}}\frac{\mathbbmsl{E}[X_SX_T]-\mathbbmsl{E}[X_S]\mathbbmsl{E}[X_T]}{\mathbbmsl{E}[X]^2}\\
&=\mathlarger{\mathlarger{\sum}}_{i=0}^\ell\,\mathlarger{\mathlarger{\sum}}^{}_{{S, T\subseteq V(G)}\atop{\mathlarger{{|S|=|T|=\ell}}\atop{\mathlarger{|S\cap T|=i}}}}\frac{\mathbbmsl{E}[X_SX_T]-\mathbbmsl{E}[X_S]\mathbbmsl{E}[X_T]}{\mathbbmsl{E}[X]^2}\\
&\leqslant\mathsmaller{{{n}\choose{\ell}}}\mathlarger{\mathlarger{\sum}}_{i=0}^\ell\,\,\frac{{{\ell}\choose{i}}
{{n-\ell}\choose{\ell-i}}\left(q^{2{{\ell}\choose{2}}-{{i}\choose{2}}}-q^{2{{\ell}\choose{2}}}\left(1-\tfrac{1}{n-\ell}\right)^2\right)}
{{{n}\choose{\ell}}^2q^{2{{\ell}\choose{2}}}\left(1-\tfrac{1}{n-\ell}\right)^2}\\
&=\mathlarger{\mathlarger{\sum}}_{i=0}^\ell\,\,
\frac{{{\ell}\choose{i}}{{n-\ell}\choose{\ell-i}}}{{{n}\choose{\ell}}q^{{i}\choose{2}}}\,\,\frac{1-q^{{i}\choose{2}}\left(1-\tfrac{1}{n-\ell}\right)^2}
{\left(1-\tfrac{1}{n-\ell}\right)^2}\\
&\leqslant\frac{{{n-\ell}\choose{\ell}}+\ell{{n-\ell}\choose{\ell-1}}}{{{n}\choose{\ell}}}\,\,
\frac{1-\left(1-\tfrac{1}{n-\ell}\right)^2}
{\left(1-\tfrac{1}{n-\ell}\right)^2}+2\sum_{i=2}^\ell
\frac{{{\ell}\choose{i}}{{n-\ell}\choose{\ell-i}}}{{{n}\choose{\ell}}q^{{i}\choose{2}}}\\
&\leqslant(\ell+1)\frac{2n-2\ell-1}{(n-\ell-1)^2}+2\sum_{i=2}^\ell
\frac{{{\ell}\choose{i}}{{n-\ell}\choose{\ell-i}}}{{{n}\choose{\ell}}q^{{i}\choose{2}}}.
\end{align*}
Using  the computations given in the proof of Theorem\,7.3 of   \cite{fri}, the last summation above converges to $0$ as $n\rightarrow\infty$ and  hence   $\mathbbmsl{P}(X=0)=o(1)$. This shows that  with high probability    there is $S\subseteq V(G)$ with $|S|=\ell$ such that $S$ is an  independent set in $G$  and $G[V(G)\setminus S]$ has a copy $L'$  of $L$ as a subgraph. Denote the spanning subgraph of $G$ with   edge set $E(L')$ by $H$. It is easily seen that $H$ is a $K_{1, r}$-saturated subgraph of $G$ and $$E(H)=\tfrac{(n-\ell)(r-1)}{2}\leqslant\tfrac{(r-1)n}{2}-(1-\epsilon)(r-1)\log_bn,$$ which  concludes \eqref{1}, as required.
\end{proof}

\section*{Acknowledgments}

The authors would like  to thank     the  anonymous referees for their  helpful comments and   corrections  on  a draft  version of this paper.

\end{document}